\documentclass[smallextended,envcountsect]{svjour3}
\usepackage{pgf,tikz}
\usepackage{mathrsfs}
\usetikzlibrary{arrows}
\smartqed
\usepackage{latexsym, amsmath, amsfonts, amscd, amssymb, verbatim, cite,marvosym}
\journalname{Optim Lett}

\begin{document}
	
	
	\title{A note on approximate Karush--Kuhn--Tucker conditions in  locally Lipschitz multiobjective optimization	
	}
	
	\titlerunning{A note on AKKT conditions in locally Lipschitz multiobjective optimization}
	\author{Nguyen Van Tuyen$^{1,2}$ \and Jen-Chih Yao$^3$ \and
	Ching-Feng Wen$^{4,5}$	
	}
	\authorrunning{N.V. Tuyen et al.}
	
	\institute{
		\Letter\ \ Ching-Feng Wen
		 \\
		{cfwen@kmu.edu.tw}
		\\
		\\
		Nguyen Van Tuyen
		\\
		{tuyensp2@yahoo.com; nguyenvantuyen83@hpu2.edu.vn}
		\\
		\\
		Jen-Chih Yao
		\\
		{yaojc@mail.cmu.edu.tw}
		\\
		\\
		\at$^1$ School of Mathematical Sciences, University of Electronic Science and Technology of China, Chengdu, P.R. China 
		\and
		\at$^2$ Department of Mathematics, Hanoi Pedagogical University 2, Xuan Hoa, Phuc Yen, Vinh Phuc, Vietnam
		\and
		\at$^3$ Center for General Education, China Medical University, Taichung, 40402, Taiwan 
		\and
		\at$^4$ Center for Fundamental Science; and Research Center for Nonlinear Analysis and Optimization, Kaohsiung Medical University, Kaohsiung, 80708, Taiwan
		\and
		\at$^5$ Department of Medical Research, Kaohsiung Medical University Hospital, Kaohsiung, 80708, Taiwan
	}

	\date{Received: date / Accepted: date}
	\maketitle
	\begin{abstract}
		In the recent paper of Giorgi, Jim\'enez and Novo (J Optim Theory Appl 171:70--89, 2016), the authors introduced the so-called approximate Karush--Kuhn--Tucker (AKKT) condition for smooth multiobjective optimization problems  and obtained some AKKT-type  necessary optimality conditions and sufficient optimality conditions for weak efficient solutions of such a problem. In this note, we extend these optimality conditions to locally Lipschitz  multiobjective optimization problems using Mordukhovich subdifferentials. Furthermore, we prove that, under some suitable additional conditions, an AKKT condition is also a KKT one.  	
	\end{abstract}

\section{Introduction}

Karush--Kuhn--Tucker (KKT) optimality conditions are one of the most important results in optimization theory. However, KKT optimality conditions do not need to be fulfilled at local minimum points unless some constraint qualifications are satisfied; see, for example, \cite[p. 97]{Andreani16}, \cite[Section 3.1]{Birgin2014} and \cite[p. 78]{Mangasarian69}. In \cite{Andreani2010}, Andreani, Mart\'inez and Svaiter introduced the so-called complementary approximate Karush--Kuhn--Tucker (CAKKT) condition for scalar optimization problems with smooth data. Then, the authors proved that this condition is necessary for a point to be a local minimizer without any constraint qualification. Moreover, they also showed that  the augmented Lagrangian method with lower-level constraints introduced in \cite{Andreani2007} generates sequences converging to CAKKT points under certain conditions. Optimality conditions of  CAKKT-type have been recognized to be useful in designing algorithms for finding approximate solutions of optimization problems; see, for example, \cite{Andreani2011,Birgin2014,Dutta2013,Haeser11,Haeser2015}.

\medskip
Recently, Giorgi, Jim\'enez and Novo \cite{Giorgi16} extended the results in \cite{Andreani2010} to multiobjective problems with continuously differentiable data.  The authors proposed the so-called approximate Karush--Kuhn--Tucker (AKKT) condition for multiobjective optimization problems. Then, they proved that the AKKT condition holds for local weak efficient solutions without any additional requirement.  Under the convexity of the related functions, an AKKT-type sufficient condition for global weak efficient solutions is also established.

\medskip
An interesting question arises: {\em How does one obtain AKKT-type optimality conditions  for weak efficient solutions of locally Lipschitz multiobjective optimization problems?}  This paper is aimed at solving the problem. We hope that our results will be useful in finding approximate efficient solutions of nonsmooth multiobjective optimization problems.

\medskip
The paper is organized as follows.  In Section \ref{Preliminaries}, we recall some basic  definitions and preliminaries from variational analysis, which are widely used in the  sequel.  Section \ref{Main results}  is devoted to presenting the main results.

\section{Preliminaries} \label{Preliminaries}
We use the following notation and terminology. Fix $n \in {\mathbb{N}}:=\{1, 2, \ldots\}$. The space $\mathbb{R}^n$ is equipped with the usual scalar product and  Euclidean norm. The topological closure and the topological interior of a subset  $S$ of $\mathbb{R}^n$ are denoted, respectively, by  $\mathrm{cl}\,{S}$ and $\mathrm{int}\,{S}$. The  closed unit ball of $\mathbb{R}^n$ is denoted by $\mathbb{B}^n.$

\begin{definition}{\rm (See \cite{mor06}) Given $\bar x\in  \mbox{cl}\,S$. The set
		\begin{equation*}
		N(\bar x; S):=\{z^*\in \mathbb{R}^n:\exists
		x^k\stackrel{S}\longrightarrow \bar x, \varepsilon_k\to 0^+, z^*_k\to z^*,
		z^*_k\in {\widehat N_{\varepsilon
				_k}}(x^k; S),\ \ \forall k \in\mathbb{N}\},
		\end{equation*}
is called the {\em Mordukhovich/limiting normal cone}  of $S$ at $\bar x$, where
		\begin{equation*}
		\Hat N_\varepsilon  (x; S):= \bigg\{ {z^*  \in {\mathbb{R}^n} \;:\;\limsup_{u\overset{S} \rightarrow x}
			\frac{{\langle z^* , u - x\rangle }}{{\parallel u - x\parallel }} \leqq \varepsilon } \bigg\}
		\end{equation*}
		is the set of  {\em $\varepsilon$-normals} of $S$  at $x$ and  $u\xrightarrow {{S}} x$ means that $u \rightarrow x$ and $u \in S$.
	}
\end{definition}

\medskip
Let  $\varphi \colon \mathbb{R}^n \to  \overline{\mathbb{R}}$ be an {\em extended-real-valued function}. The {\em  epigraph}  and  {\em domain} of $\varphi$ are denoted, respectively, by
\begin{align*}
\mbox{epi }\varphi&:=\{(x, \alpha)\in\mathbb{R}^n\times\mathbb{R} \,:\,  \alpha\geqq \varphi(x) \},
\\
\mbox{dom }\varphi &:= \{x\in \mathbb{R}^n \,:\, \ \ |\varphi(x)|<+\infty \}.
\end{align*}

\begin{definition}{\rm  (See \cite{mor06})
		Let $\bar x\in \mbox{dom }\varphi$. The set
		\begin{align*}
		\partial \varphi (\bar x):=\{x^*\in \mathbb{R}^n \,:\, (x^*, -1)\in N((\bar x, \varphi (\bar x)); \mbox{epi }\varphi )\},
		\end{align*}
		is called the {\it Mordukhovich/limiting subdifferential}  of $\varphi$ at $\bar x$. If $\bar x\notin \mbox{dom }\varphi$, then we put $\partial \varphi (\bar x)=\emptyset$.
	}
\end{definition}

\medskip
Recall that $\varphi : \mathbb{R}^n\to\mathbb{R}^m$ is {\em strictly differentiable} at $\bar x$  iff there is a linear continuous operator $\nabla\varphi(\bar x) : \mathbb{R}^n\to \mathbb{R}^m$, called the {\em Fr\'echet derivative} of $\varphi$ at $\bar x$, such that
\begin{equation*}
\lim_{\substack{x\to \bar x\\ u\to \bar x}}\dfrac{\varphi(x)-\varphi(u)-\nabla \varphi(\bar x)( x-u)}{\|x-u\|}=0.
\end{equation*}
As is well-known, any function $\varphi$ that is continuously differentiable in a neighborhood of $\bar x$ is strictly differentiable at $\bar x$. We now summarize some properties of  the Mordukhovich subdifferential that will be used in the next section.

\begin{proposition}\label{nonnegative-scaling}{\rm (See \cite[Proposition 6.17(d)]{Penot13})} Let $\varphi\colon\mathbb{R}^n\to\overline{\mathbb{R}}$ be lower semicontinuous around $\bar x$. Then, for all $\lambda\geqq 0$, one has $\partial (\lambda\varphi)(\bar x)=\lambda \partial \varphi(\bar x)$.
\end{proposition}

\begin{proposition}\label{compactness}{\rm (See \cite[Corollary 1.81]{mor06})} If $\varphi\colon\mathbb{R}^n\to\overline{\mathbb{R}}$ is Lipschitz continuous
	around $\bar x$ with modulus $L>0$, then $\partial \varphi(\bar x)$ is a nonempty compact set in $\mathbb{R}^n$ and contained in $L\mathbb{B}^n$.	
\end{proposition}

\begin{proposition}\label{sum-rule}{\rm (See \cite[Theorem 3.36]{mor06})} Let $\varphi_l\colon\mathbb{R}^n\to\overline{\mathbb{R}}$, $l=1, \ldots, p$, $p\geqq 2$, be lower semicontinuous around $\bar x$ and let all but one of these
	functions be locally Lipschitz around $\bar x$. Then we have the following inclusion
	\begin{equation*}
	\partial (\varphi_1+\ldots+\varphi_p) (\bar x)\subset \partial  \varphi_1 (\bar x) +\ldots+\partial \varphi_p (\bar x).
	\end{equation*}
\end{proposition}

\begin{proposition}\label{max-rule}{\rm (See \cite[Theorem 3.46]{mor06})}
	Let $\varphi_l\colon\mathbb{R}^n\to\overline{\mathbb{R}}$, $l=1, \ldots, p$,  be  locally Lipschitz around $\bar x$. Then the function
	$ \phi(\cdot):=\max\{\varphi_l(\cdot):l=1, \ldots, p\}$
	is also locally Lipschitz around $\bar x$ and one has
	\begin{equation*}
	\partial \phi(\bar x)\subset \bigcup\bigg\{\partial\bigg(\sum_{l=1}^{p}\lambda_l\varphi_l\bigg)(\bar x)\;:\; (\lambda_1, \ldots, \lambda_p)\in\Lambda(\bar x)\bigg\},
	\end{equation*}
	where $\Lambda(\bar x):=\big\{(\lambda_1, \ldots, \lambda_p)\;:\; \lambda_l\geqq 0, \sum_{l=1}^{p}\lambda_l=1, \lambda_l[\varphi_l(\bar x)-\phi(\bar x)]=0\big\}.$
\end{proposition}

\begin{proposition}\label{chain-rule}{\rm (See \cite[Theorem 3.41]{mor06})}  Let $g\colon \mathbb{R}^n\to\mathbb{R}^m$ be locally Lipschitz around $\bar x$ and $\varphi\colon\mathbb{R}^m\to \mathbb{R}$ be locally Lipschitz around $g(\bar x)$. Then one has
	\begin{equation*}
	\partial (\varphi\circ g)(\bar x)\subset \bigcup_{y\in \partial \varphi(g(\bar x))} \partial \langle y, g\rangle (\bar x).
	\end{equation*}
	In particular, if $m=1$ and $\varphi$ is strictly differentiable at $g(\bar x)$, then
	\begin{equation*}
	\partial (\varphi\circ g)(\bar x)\subset \partial(\nabla \varphi(g(\bar x))g)(\bar x).
	\end{equation*}	
\end{proposition}

\begin{proposition}\label{Fermat-rule}{\rm (See \cite[Proposition 1.114]{mor06})} Let   $\varphi\colon\mathbb{R}^n\to\overline{\mathbb{R}}$  be finite at $\bar x$. If $\varphi$ has a local minimum at $\bar x$, then $ 0\in\partial\varphi(\bar x).$
\end{proposition}
\begin{proposition}\label{robustness}{\rm (See \cite[Proposition 5.2.28]{Borwein05})}  Let  $\varphi\colon\mathbb{R}^n\to\overline{\mathbb{R}}$ be a lower semicontinuous function. Then the set-valued mapping $\partial \varphi\colon\mathbb{R}^n\rightrightarrows\mathbb{R}^n$ is closed.
\end{proposition}
\section{Main results}\label{Main results}
Let $\mathcal{L}:=\{1, \ldots, p\}$, $\mathcal{I}:=\{1, \ldots, m\}$ and $\mathcal{J}:=\{1, \ldots, r\}$ be index sets.  Suppose that  $f=(f_1, \ldots, f_p)\colon \mathbb{R}^n\to\mathbb{R}^p$,  $g=(g_1, \ldots, g_m)\colon \mathbb{R}^n\to\mathbb{R}^m$, and $h=(h_1, \ldots, h_r)\colon \mathbb{R}^n\to\mathbb{R}^r$ are vector-valued functions with locally Lipschitz components  defined on $\mathbb{R}^n$. Let $\mathbb{R}^p_+$ be the nonnegative orthant of $\mathbb{R}^p$.  For $a, b\in \mathbb{R}^p$, by $a\leqq b$, we mean $a-b\in -\mathbb{R}^p_+$; by $a\leq b$, we mean $a-b\in -\mathbb{R}^p_+\setminus\{0\}$; and by $a<b$, we mean $a-b\in-\text{int}\,\mathbb{R}^p_+$.

\medskip
We focus on the following constrained multiobjective optimization problem:
\begin{align}\label{problem}
\min\,_{\mathbb{R}^p_+} f(x)\ \ \text{subject to}\ \ x\in \mathcal{F}, \tag{MOP}
\end{align}
where $\mathcal{F}$  is the feasible set given by
$
\mathcal{F}:=\{x\in\mathbb{R}^n\,:\, g(x)\leqq 0, h(x)=0\}.
$

\begin{definition}{\rm Let $\bar x\in \mathcal{F}$. We say that:
		\begin{enumerate}
			\item [(i)] $\bar x$ is {\em a (global) weak efficient solution} of \eqref{problem}  iff there is no $x\in \mathcal{F}$ satisfying   $f(x)<f(\bar x)$.
			
			\item [(ii)] $\bar x$ is a {\em local weak efficient solution} of \eqref{problem} iff there exists a neighborhood $U$  of $\bar x$ such that $\bar x$ is a weak efficient solution on $U\cap \mathcal{F}$.
		\end{enumerate}
	}
\end{definition}

\medskip
We now introduce the concept of approximate Karush--Kuhn--Tucker condition for \eqref{problem} inspired by the work of Giorgi, Jim\'enez and Novo \cite{Giorgi16}.

\begin{definition}{\rm We say that the {\em approximate Karush--Kuhn--Tucker condition} (AKKT) is satisfied for \eqref{problem} at a feasible point $\bar x$ iff there exist sequences $\{x^k\}\subset \mathbb{R}^n$ and $\{(\lambda^k, \mu^k, \tau^k)\}\subset \mathbb{R}^p_+\times\mathbb{R}^m_+\times\mathbb{R}^r$ such that
		\begin{enumerate}
			\item [(A0)] $x^k\to \bar x$,
			\item[(A1)] $\mathfrak{m}(x^k; \lambda^k, \mu^k, \tau^k)\to 0$ as $k\to \infty$, where
			\begin{align*}
			\mathfrak{m}(x^k; \lambda^k, \mu^k, \tau^k):=\inf\bigg\{&\bigg\|\sum_{l=1}^{p}\lambda_l^k\xi_l+\sum_{i=1}^{m}\mu^k_i\eta_i+\sum_{j=1}^{r}\tau^k_j\gamma_j\bigg\|\;:\; \xi_l\in \partial f_l(x^k),
			\\
			&\eta_i\in \partial g_i(x^k), \gamma_j\in [\partial h_j(x^k)\cup\partial(-h_j)(x^k)]\bigg\},
			\end{align*}
			\item[(A2)] $\displaystyle\sum_{l=1}^{p}\lambda^k_l=1$,
			\item[(A3)] $g_i(\bar x)<0\Rightarrow \mu^k_i=0$ for sufficiently large $k$ and $i\in\mathcal{I}$.
		\end{enumerate}	
	}	
\end{definition}

\medskip
We are now ready to state and prove our main results.

\begin{theorem}\label{necessary-theorem} If $\bar x\in \mathcal{F}$ is a local weak efficient solution of \eqref{problem}, then there exist sequences $\{x^k\}$ and $\{(\lambda^k, \mu^k, \tau^k)\}$ satisfying the AKKT condition at $\bar x$. Furthermore, we can choose these sequences such that the following conditions hold:
	\begin{enumerate}
		\item [{\rm(E1)}] $\mu_i^k=b_k\max(g_i(x^k), 0)\geqq 0,$  $\forall i\in\mathcal{I}$, and  $\tau_j^k=c_kh_j(x^k)\geqq 0,$   $\forall j\in\mathcal{J}$, where $b_k, c_k>0,$  $\forall k\in\mathbb{N}$,
		\item [{\rm(E2)}] $f_l(x^k)-f_l(\bar x)+\frac{1}{2}\left[\sum_{i=1}^{m}\mu_i^kg_i(x^k)+ \sum_{j=1}^{r}\tau^k_jh_j(x^k)\right]\leqq 0,$  $\forall k\in\mathbb{N}$, $l\in\mathcal{L}$.
	\end{enumerate}
\end{theorem}
\begin{proof} Since $\bar x$ is a local weak efficient solution of \eqref{problem}, $f_l$, $g_i$ and $h_j$ are locally Lipschitz functions, we can choose $\delta>0$ such that these functions are Lipschitz on $B(\bar x, \delta):=\{x\in\mathbb{R}^n\;:\; \|x-\bar x\|\leqq \delta\}$ and $\bar x$ is a global weak efficient solution of $f$ on $\mathcal{F}\cap B(\bar x, \delta)$. It is easily seen that $\bar x$ is also a global minimum solution of the function $\phi(\cdot):=\max\{f_l(\cdot)-f_l(\bar x)\;:\; l\in\mathcal{L}\}$ on $\mathcal{F}\cap B(\bar x, \delta)$.

For each $k\in \mathbb{N}$, we consider the following problem
	\begin{align}\label{penalized-problem}
	\min \{\varphi_k(x)\;:\; x\in B(\bar x, \delta)\}, \tag{P$_k$}
	\end{align}
	where
	$$\varphi_k(x):=\phi(x)+\frac{k}{2}\bigg[\sum_{i=1}^{m}(\max(g_i(x),0))^2+\sum_{j=1}^{r}(h_j(x))^2\bigg]+\frac{1}{2}\|x-\bar x\|^2.$$
	Clearly, $\varphi_k$ is continuous on the compact set $B(\bar x, \delta)$. Thus, by the Weierstrass theorem, the problem \eqref{penalized-problem} admits an optimal solution, say $x^k$. This and the fact that $\varphi_k(\bar x)=0$ imply that
	\begin{equation} \label{equa-1}
	\phi(x^k)+\frac{k}{2}\bigg[\sum_{i=1}^{m}(\max(g_i(x^k),0))^2+\sum_{j=1}^{r}(h_j(x^k))^2\bigg]+\frac{1}{2}\|x^k-\bar x\|^2\leqq 0,
	\end{equation}
	or, equivalently,
	\begin{equation}\label{equa-2}
	\bigg[\sum_{i=1}^{m}(\max(g_i(x^k),0))^2+\sum_{j=1}^{r}(h_j(x^k))^2\bigg]\leqq -\frac{1}{k}\bigg[2\phi(x^k)+\|x^k-\bar x\|^2\bigg].
	\end{equation}
	By the continuity of $\phi$ and $\|x^k-\bar x\|\leqq \delta$, the right-hand-side of \eqref{equa-2} tends to zero when $k$ tends to infinity. Hence,
	\begin{equation*}
	\max(g_i(x^k),0)\to 0, \ \ \forall i\in\mathcal{I},\ \ h_j(x^k) \to 0, \ \ \forall j\in\mathcal{J}, \ \ \text{as}\ \ k\to\infty.
	\end{equation*}
This and the continuity of the functions $\max(g_i(\cdot),0)$ and $h_j$ imply that every accumulation point of $\{x^k\}$ must belongs to $\mathcal{F}$. Since $\{x^k\}\subset B(\bar x, \delta)$, the sequence has at least an accumulation point, say $\tilde{x}\in \mathcal{F}$. By \eqref{equa-1}, one has
	\begin{equation*}
	\phi(x^k)+\frac{1}{2}\|x^k-\bar x\|^2\leqq 0, \ \ \forall k\in\mathbb{N}.
	\end{equation*}
	Passing the last inequality to the limit as $k\to\infty$, we get
	\begin{equation*}
	\phi(\tilde{x})+\frac{1}{2}\|\tilde{x}-\bar x\|^2\leqq 0.
	\end{equation*}
	This and $\phi(\tilde{x})\geqq 0$ imply that $\tilde{x}=\bar x$. This means that the sequence $\{x^k\}$ has a unique accumulation point $\bar x$, thus
	converges. Consequently, $x^k$ belongs to the interior of $B(\bar x, \delta)$ for $k$ large enough. Thanks to Proposition \ref{Fermat-rule}, we have
	\begin{equation}\label{equa-3}
	0\in \partial \varphi_k(x^k).
	\end{equation}
	By Propositions \ref{nonnegative-scaling}--\ref{chain-rule}, one has
	\begin{equation*}
	\partial \varphi_k(x^k) \subset \partial \phi(x^k)+\frac{k}{2} \sum_{i=1}^{m}\partial(\max(g_i(\cdot),0))^2(x^k)+\frac{k}{2}\sum_{j=1}^{r}\partial(h_j)^2(x^k)+(x^k-\bar x),
	\end{equation*}
	where
	\begin{align*}
	\partial \phi(x^k)&\subset \bigcup\bigg\{\partial\bigg(\sum_{l=1}^{p}\lambda_l(f_l(\cdot)-f_l(\bar x))\bigg)(x^k)\;:\; (\lambda_1, \ldots, \lambda_p)\in\Lambda(x^k)\bigg\}
	\\
	&\subset \bigcup\bigg\{ \sum_{i=1}^{p}\lambda_l\partial f_l (x^k)\;:\; (\lambda_1, \ldots, \lambda_p)\in\Lambda(x^k)\bigg\},
	\end{align*}
	with $$\Lambda(x^k)=\bigg\{(\lambda_1, \ldots, \lambda_p):\lambda_l\geqq 0, \sum_{l=1}^{p}\lambda_l=1, \lambda_l[(f_l(x^k)-f_l(\bar x))-\phi(x^k)]=0\bigg\},$$
	and
	\begin{align*}
	\partial(\max(g_i(\cdot),0))^2(x^k)&\subset \partial (2\max(g_i(x^k), 0)\max (g_i(\cdot), 0))(x^k)
	\\
	&=2\max(g_i(x^k), 0)\,\partial (\max (g_i(\cdot), 0))(x^k)
		\\
	&= 2\max(g_i(x^k), 0)\,\partial g_i(x^k),
	\\
	\partial(h_j)^2(x^k)&\subset \partial (2h_j(x^k)h_j)(x^k)
		\\
	&\subset 2|h_j(x^k)|\left[\partial h_j(x^k)\cup\partial (-h_j)(x^k)\right].
	\end{align*}
	Hence, \eqref{equa-3} implies that
	\begin{align*}
	0&\in  \bigcup\bigg\{ \sum_{i=1}^{p}\lambda_l\partial f_l (x^k)\;:\; (\lambda_1, \ldots, \lambda_p)\in\Lambda(x^k)\bigg\} +k\sum_{i=1}^{m}\max(g_i(x^k), 0) \partial g_i(x^k)
	\\
	&\;\quad\quad\quad\quad\quad\quad\quad\quad\quad\quad+k\sum_{j=1}^{r} |h_j(x^k)|\left[\partial h_j(x^k)\cup\partial (-h_j)(x^k)\right]+(x^k-\bar x).
	\end{align*}
	This means that there exist $(\lambda_1^k, \ldots, \lambda_p^k)\in \Lambda(x^k)$, $\xi_l^k\in \partial f_l (x^k)$, $\eta_i^k\in \partial g_i(x^k)$ and $\gamma_j^k\in \left[\partial h_j(x^k)\cup\partial (-h_j)(x^k)\right]$  such that
	\begin{equation*}
	\sum_{i=1}^{p}\lambda_l^k\xi_l^k+k\sum_{i=1}^{m}\max(g_i(x^k), 0)\eta_i^k+k\sum_{j=1}^{r} |h_j(x^k)|\gamma_j^k+(x^k-\bar x)=0.
	\end{equation*}
	Hence,
	\begin{equation*}
	\bigg\|\sum_{i=1}^{p}\lambda_l^k\xi_l^k+k\sum_{i=1}^{m}\max(g_i(x^k), 0)\eta_i^k+k\sum_{j=1}^{r} |h_j(x^k)|\gamma_j^k\bigg\|=\|x^k-\bar x\|.
	\end{equation*}
	Setting
	$	\lambda^k=(\lambda_1^k, \ldots, \lambda_p^k), \mu^k=(\mu_1^k, \ldots, \mu_m^k), \tau^k=(\tau_1^k, \ldots, \tau_r^k),$ 	where
	\begin{align*}
	\mu_i^k:=k\max(g_i(x^k), 0)\geqq 0, \ \ \forall i\in\mathcal{I}, 	\tau_j^k:=k|h_j(x^k)|\geqq 0, \ \ \forall j\in\mathcal{J}.
	\end{align*}
	For each $k\in\mathbb{N}$, we have
	\begin{align*}
	0\leqq \mathfrak{m}(x^k; \lambda^k, \mu^k, \tau^k)&\leqq \bigg\|\sum_{i=1}^{p}\lambda_l^k\xi_l^k+\sum_{i=1}^{m}\mu_i^k\eta_i^k+\sum_{j=1}^{r} \tau_j^k\gamma_j^k\bigg\|
	\\
	&=\|x^k-\bar x\|.
	\end{align*}
This and $\lim\limits_{k\to\infty}x^k=\bar x$ imply that $\lim\limits_{k\to\infty}\mathfrak{m}(x^k; \lambda^k, \mu^k, \tau^k)=0.$ Thus, $\bar x$ satisfies conditions (A0)--(A2). If $g_j(\bar x)<0$, then $g_j(x^k)<0$ for  $k$ large enough. Consequently, $\mu_i^k=0$ for $k$ large enough and we therefore get condition (A3).
	
For each $j\in\mathcal{J}$, by passing to a subsequence if necessary, we may assume that $h_j(x^k)\geqq 0$ for all $k\in\mathbb{N}$, or $h_j(x^k)< 0$ for all $k\in\mathbb{N}$. For the last case, by replacing $h_j$ by $\bar h_j:=-h_j$, one has
	\begin{equation*}
	\{x\in\mathbb{R}^n:g_i(x)\leqq 0, i\in\mathcal{I}, h_k(x)=0, k\in\mathcal{J}, k\neq j, \bar h_j(x)=0\}=\mathcal{F},
	\end{equation*}
	\begin{equation*}
	\partial h_j(x^k)\cup\partial (-h_j)(x^k)=\partial \bar h_j(x^k)\cup\partial (-\bar h_j)(x^k),
	\end{equation*}
	and $\bar h_j(x^k)\geqq 0$ for all $k\in\mathbb{N}$. Hence we may assume that  $h_j(x^k)\geqq 0$ for all $k\in\mathbb{N}$ and $j\in \mathcal{J}$. This means that  $\tau^k_j=kh_j(x^k)\geqq 0$ for all $k\in\mathbb{N}$ and $j\in\mathcal{J}$ and we therefore get condition (E1). Moreover, we see that
	\begin{equation*}
	\mu_i^kg_i(x^k)=k(\max(g_i(x^k), 0))^2\ \ \text{and}\ \ \tau^k_jh_j(x^k)=k(h_j(x^k))^2.
	\end{equation*}
	Thus, \eqref{equa-1} can be rewrite as
	\begin{equation*}
	\phi(x^k)+\frac{1}{2}\bigg[\sum_{i=1}^{m}\mu_i^kg_i(x^k)+\sum_{j=1}^{r}\tau_j^kh_j(x^k)\bigg]+\frac{1}{2}\|x^k-\bar x\|^2\leqq 0
	\end{equation*}
	and condition (E2) follows. The proof is complete. $\hfill\Box$
\end{proof}

\begin{remark}{\rm If $h_j$, $j\in\mathcal{J}$, are continuously differentiable functions, then
		\begin{align*}
		&\partial (h_j)^2(x^k)=2h_j(x^k)\nabla h_j(x^k) \ \ \text{and}
		\\
		&\partial h_j(x^k)\cup\partial (-h_j)(x^k)= \left\{\nabla h_j(x^k), -\nabla h_j(x^k) \right\}.
		\end{align*}}
In this case, we can choose $\gamma_j^k=\nabla h_j(x^k)$ for all  $j\in\mathcal{J}$ and $k\in\mathbb{N}$. Thus, the conclusions of Theorem \ref{necessary-theorem} still hold if condition (A1) is replaced by the following condition:
\begin{enumerate}
	\item [(A1)$^\prime$] $\mathfrak{m}^\prime(x^k; \lambda^k, \mu^k, \tau^k)\to 0$ as $k\to \infty$, where
	\begin{align*}
	\mathfrak{m}^\prime(x^k; \lambda^k, \mu^k, \tau^k):=\inf\bigg\{&\bigg\|\sum_{l=1}^{p}\lambda_l^k\xi_l+\sum_{i=1}^{m}\mu^k_i\eta_i+\sum_{j=1}^{r}\tau^k_j\nabla h_j(x^k)\bigg\|\;:\;
	\\
	&\quad\quad\quad\quad\quad\quad\quad\quad\xi_l\in \partial f_l(x^k), \eta_i\in \partial g_i(x^k)\bigg\}.
	\end{align*}
\end{enumerate}
Conditions (A0), (A1)$^\prime$, (A2), (A3) are called by the AKKT$^\prime$ condition. In case the problem \eqref{problem} has no equality constraints, then conditions  (A1) and (A1)$^\prime$ coincide. In general, condition (A1)$^\prime$ is stronger than condition (A1) because
\begin{equation*}
\mathfrak{m}(x^k; \lambda^k, \mu^k, \tau^k)\leqq\mathfrak{m}^\prime(x^k; \lambda^k, \mu^k, \tau^k).
\end{equation*}
Thus if $\bar x$ satisfies the AKKT$^\prime$ condition with respect to sequences $\{x^k\}$ and $\{(\lambda^k, \mu^k, \tau^k)\}$, then so does the AKKT one.
\end{remark}

\begin{definition} {\rm (See \cite[Remark 3.2]{Giorgi16}) Let $\bar x$ be a feasible point of \eqref{problem}. We say that:
\begin{enumerate}
	\item [(i)] $\bar x$ satisfies the {\em  sign condition} (SGN) with respect to sequences $\{x^k\}\subset \mathbb{R}^n$ and $\{(\lambda^k, \mu^k, \tau^k)\}\subset \mathbb{R}^p_+\times\mathbb{R}^m_+\times\mathbb{R}^r$ iff, for every $k\in\mathbb{N}$, one has
	\begin{equation*}
	\mu_i^k g_i(x^k)\geqq 0, i\in\mathcal{I}, \ \ \text{and} \ \ \tau_j^k h_j(x^k)\geqq 0, j\in\mathcal{J}.
	\end{equation*}
	\item [(ii)] $\bar x$ satisfies the {\em  sum converging to zero condition} (SCZ) with respect to sequences $\{x^k\}\subset \mathbb{R}^n$ and $\{(\lambda^k, \mu^k, \tau^k)\}\subset \mathbb{R}^p_+\times\mathbb{R}^m_+\times\mathbb{R}^r$ iff
	\begin{equation*}
	\sum_{i=1}^{m}\mu_i^kg_i(x^k)+\sum_{j=1}^{r} \tau_j^kh_j(x^k)\to 0 \ \ \text{as} \ \ k\to\infty.
	\end{equation*}
\end{enumerate}
	}
\end{definition}
\begin{remark}{\rm Clearly, if condition (E1) holds at $\bar x$, then so does condition SGN. Moreover, thanks to \cite[Remark 3.2]{Giorgi16}, conditions (A0), SGN and (E2) imply condition SCZ. The converse does not hold in general; see \cite[Remark 3.4]{Giorgi16}.		
	}
\end{remark}

\medskip
The following result gives sufficient optimality conditions for (global) weak efficient solutions of convex problems.

\begin{theorem}\label{sufficient-theorem} Assume that $f_l$ $(l=1, \ldots, p)$ and $g_i$ $(i=1, \ldots, m)$ are convex and $h_j$ $(j=1, \ldots, r)$ are affine. If $\bar x$ satisfies  conditions $AKKT^\prime$ and SCZ with respect to sequences $\{x^k\}\subset \mathbb{R}^n$ and $\{(\lambda^k, \mu^k, \tau^k)\}\subset \mathbb{R}^p_+\times\mathbb{R}^m_+\times\mathbb{R}^r$, then $\bar x$ is a  weak efficient solution of \eqref{problem}.
\end{theorem}
\begin{proof} On the contrary, suppose that $\bar x$ is not a weak efficient solution of \eqref{problem}. Then, there exists $\hat x\in\mathcal{F}$ such that
	\begin{equation}\label{equa-4}
	f_l(\hat x)<f_l(\bar x) \ \ \text{for all} \ \ l\in\mathcal{L}.
	\end{equation}
	By condition (A2), without any loss of generality, we may assume that $\lambda^k\to \lambda$ with $\lambda\geq 0$ and $\displaystyle\sum_{l=1}^{p}\lambda_l=1$.  For $k$ large enough, the sets  $\partial f_l(x^k)$ and $\partial g_i(x^k)$ are compact. Hence, there exist  $\xi^k_l\in\partial f_l(x^k)$ and $\eta^k_i\in\partial g_i(x^k)$  such that
	\begin{equation*}
	\mathfrak{m}^\prime(x^k; \lambda^k, \mu^k, \tau^k)=\bigg\|\sum_{l=1}^{p}\lambda_l^k\xi^k_l+\sum_{i=1}^{m}\mu^k_i\eta^k_i+\sum_{j=1}^{r}\tau^k_j\nabla h_j(x^k)\bigg\|.
	\end{equation*}
As $f_l$ and  $g_i$ are convex and $h_j$ are affine, for each $k\in\mathbb{N}$, we have
	\begin{align}
	f_l(\hat x)&\geqq f_l(x^k)+\langle \xi_l^k, \hat x-x^k\rangle, \ \ \forall l\in\mathcal{L},\label{equa-5}
	\\
	g_i(\hat x)&\geqq g_i(x^k)+\langle \eta_i^k, \hat x-x^k\rangle, \ \ \forall i\in\mathcal{I},\label{equa-6}
	\\
	h_j(\hat x)&= h_j(x^k)+\langle \nabla h_j(x^k), \hat x-x^k\rangle, \ \ \forall j\in\mathcal{J}.\label{equa-7}
	\end{align}
	Multiplying \eqref{equa-5} by $\lambda_l^k$, \eqref{equa-6} by $\mu_i^k$, \eqref{equa-7} by $\tau^k_j$ and adding up, we obtain
	\begin{align}
	\sum_{l=1}^{p}\lambda_l^kf_l(\hat x)&\geqq  \sum_{l=1}^{p}\lambda_l^kf_l(\hat x)+\sum_{i=1}^{m}\mu_i^kg_i(\hat x)+\sum_{j=1}^{r}\tau_j^kh_j(\hat x)\notag
	\\
	&\geqq\sum_{l=1}^{p}\lambda_l^kf_l(x^k)+\sum_{i=1}^{m}\mu_i^kg_i(x^k)+\sum_{j=1}^{r}\tau_j^kh_j(x^k)+\sigma_k,\label{equa-8}
	\end{align}
	where $\sigma_k:= \big(\sum_{l=1}^{p}\lambda_l^k\xi^k_l+\sum_{i=1}^{m}\mu^k_i\eta^k_i+\sum_{j=1}^{r}\tau^k_j\nabla h_j(x^k)\big)(\hat x-x^k).$ 	Since $x^k\to \bar x$ and $\mathfrak{m}^\prime(x^k; \lambda^k, \mu^k, \tau^k)\to 0$  as $k\to \infty$, and
	\begin{align*}
	\|\sigma_k\|&\leqq \bigg\|\sum_{l=1}^{p}\lambda_l^k\xi^k_l+\sum_{i=1}^{m}\mu^k_i\eta^k_i+\sum_{j=1}^{r}\tau^k_j\nabla h_j(x^k)\bigg\| \|\hat x-x^k\|
	\\
	&= \mathfrak{m}^\prime(x^k; \lambda^k, \mu^k, \tau^k)\|\hat x-x^k\|,
	\end{align*}
	we have $\lim\limits_{k\to\infty}\sigma_k= 0$. By condition SCZ,  taking the limit in \eqref{equa-8}, we obtain
	\begin{equation}\label{equa-9}
	\sum_{l=1}^{p}\lambda_lf_l(\hat x)\geqq\sum_{l=1}^{p}\lambda_lf_l(\bar x).
	\end{equation}
	Moreover, since $\lambda\geq 0$ and \eqref{equa-4}, we have $	\sum_{l=1}^{p}\lambda_lf_l(\hat x)<\sum_{l=1}^{p}\lambda_lf_l(\bar x),$	
	contrary to \eqref{equa-9}. The proof is complete. $\hfill\Box$
\end{proof}

\medskip
Clearly, if $f$, $g$ and $h$ are continuously differentiable, then Theorem  \ref{necessary-theorem} and Theorem \ref{sufficient-theorem} reduce to  \cite[Theorem 3.1]{Giorgi16} and \cite[Theorem 3.2]{Giorgi16}, respectively.

\medskip
We now show that, under the additional that the quasi-normality constraint qualification and  condition (E1) hold at a given feasible solution $\bar x$, an AKKT condition is also a KKT one.

\begin{definition}{\rm We say that $\bar x\in\mathcal{F}$ satisfies the {\em   KKT optimality condition} iff there exists a multiplier $(\lambda, \mu, \tau)$  in $\mathbb{R}^p_+\times\mathbb{R}^m_+\times\mathbb{R}^r$ such that
		\begin{enumerate}
			\item [(i)] $\lambda \geq 0$,
			\item [(ii)] $0\in \displaystyle\sum_{l=1}^{p}\lambda_l\partial f_l(\bar x)+\sum_{i=1}^{m}\mu_i\partial g_i(\bar x)+\sum_{j=1}^{r}\tau_j[\partial h_j(\bar x)\cup \partial (-h_j)(\bar x)]$,
			\item [(iii)] $\mu_i g_i(\bar x)=0$, $i\in \mathcal{I}$.
		\end{enumerate}
		
	}
\end{definition}

\begin{definition}\label{QNCQ-def}{\rm We say that $\bar x\in\mathcal{F}$ satisfies the {\em  quasi-normality constraint qualification} (QNCQ) if there is not any multiplier $(\mu, \tau)\in \mathbb{R}^m_+\times\mathbb{R}^r$ satisfying
\begin{enumerate}
	\item [(i)] $(\mu, \tau)\neq 0$,
	\item [(ii)] $0\in \displaystyle\sum_{i=1}^{m}\mu_i\partial g_i(\bar x)+\sum_{j=1}^{r}\tau_j[\partial h_j(\bar x)\cup \partial (-h_j)(\bar x)]$,
	\item [(iii)] in every neighborhood of $\bar x$ there is a point $x\in\mathbb{R}^n$ such that $g_i(x)>0$ for all $i$ having $\mu_i> 0$, and $\tau_jh_j(x)>0$ for all $j$ having $\tau_j\neq 0$.
\end{enumerate}		
		}	
\end{definition}

\begin{theorem}\label{KKT-theorem} Let $\bar x\in\mathcal{F}$ be such that  conditions AKKT and (E1) are satisfied with respect to sequences $\{x^k\}$ and $\{(\lambda^k, \mu^k, \tau^k)\}$. If the QNCQ holds  at $\bar x$, then so does the KKT optimality condition.
\end{theorem}
\begin{proof} For each $k\in\mathbb{N}$, put $\delta_k=\|(\lambda^k, \mu^k, \tau^k)\|$. By condition (A2),  we have
\begin{equation}\label{equ-11}
\delta_k \geqq \bigg(\sum_{l=1}^{p}(\lambda_l^k)^2\bigg)^{\frac{1}{2}} \geqq\frac{1}{\sqrt{p}}>0.
\end{equation}
Since $\left\|\frac{1}{\delta_k}(\lambda^k, \mu^k, \tau^k)\right\|=1$ for all $k\in\mathbb{N}$, we may assume that the sequence $\left\{\frac{1}{\delta_k}(\lambda^k, \mu^k, \tau^k)\right\}$ converges to $(\lambda, \mu, \tau)\in (\mathbb{R}^p_+\times\mathbb{R}^m_+\times\mathbb{R}^r)\setminus\{0\}$ as $k$ tends to infinity. By condition (A0) and Proposition \ref{compactness}, for  $k$ large enough, the sets $\partial f_l(x^k)$, $\partial g_i(x^k)$ and $[\partial h_j(x^k)\cup\partial(-h_j)(x^k)]$ are compact. Thus, there exist $\xi_l^k\in \partial f_l(x^k)$, $\eta_i^k\in \partial g_i(x^k)$ and $\gamma_j^k\in [\partial h_j(x^k)\cup\partial(-h_j)(x^k)]$ such that
\begin{equation}\label{equa-10}
\mathfrak{m}(x^k; \lambda^k, \mu^k, \tau^k)=\bigg\|\sum_{l=1}^{p}\lambda_l^k\xi^k_l+\sum_{i=1}^{m}\mu^k_i\eta^k_i+\sum_{j=1}^{r}\tau^k_j\gamma_j^k\bigg\|
\end{equation}
for   $k$ large enough.  Since $f_l, g_i$ and $h_j$ are locally Lipschitz around $\bar x$, without any loss of generality, we may assume that these functions are  locally Lipschitz around $\bar x$ with the same modulus $L$. Again by condition (A0) and Proposition \ref{compactness}, for $k$ large enough, one has $ (\xi_l^k, \eta_i^k, \gamma_j^k)\in L \mathbb{B}^n\times L\mathbb{B}^n\times L\mathbb{B}^n.$ By replacing  $\{(\xi_l^k, \eta_i^k, \gamma_j^k)\}$ by a subsequence if necessary, we may assume that this sequence converges to some $(\xi_l, \eta_i, \gamma_j)\in \mathbb{R}^n\times\mathbb{R}^n\times\mathbb{R}^n$. By Proposition \ref{robustness}, we have
\begin{equation*}
(\xi_l, \eta_i, \gamma_j)\in \partial f_l(\bar x)\times\partial g_i(\bar x)\times [\partial h_j(\bar x)\cup \partial(-h_j)(\bar x)].
\end{equation*}
From conditions (A1) and \eqref{equ-11}, dividing the both sides of \eqref{equa-10} by $\delta_k$ and taking the limits, we have
\begin{equation*}
\sum_{l=1}^{p}\lambda_l\xi_l+\sum_{i=1}^{m}\mu_i\eta_i+\sum_{j=1}^{r}\tau_j\gamma_j=0.
\end{equation*}
Thanks to condition (A3), one has $\mu_i g_i(\bar x)=0 \ \ \text{for all} \ \ i\in \mathcal{I}$. We claim that $\lambda\neq 0$. Indeed, if otherwise, one has $(\mu, \tau)\neq 0$ and
\begin{equation*}
\sum_{i=1}^{m}\mu_i\eta_i+\sum_{j=1}^{r}\tau_j\gamma_j=0.
\end{equation*}
By condition \eqref{equ-11} and $\mu_i^k\to \mu_i$ as $k\to\infty$, we see that if $\mu_i>0$, then $\mu_i^k>0$ for $k$ large enough. Hence, due to condition (E1), we obtain  $g_i(x^k)>0$ for all $k$ large enough. Similarly, if $\tau_j\neq 0$, then $\tau_j h_j(x^k)>0$ for $k$ large enough. Thus, the multiplier $(\mu, \tau)$ satisfies conditions (i)--(iii) in Definition \ref{QNCQ-def}, contrary to the fact that $\bar x$ satisfies the QNCQ. The proof is complete. $\hfill\Box$
\end{proof}

\medskip
We finish this section with the following remarks.

\begin{remark}
	\begin{enumerate}
		\item [(i)]  It is well known that if $\bar x$ is a weak efficient solution of \eqref{problem} and satisfies  the QNCQ, then the KKT condition holds at this point; see \cite[Theorem 3.3]{Chuong14}. This fact may not hold if $\bar x$ is not a  weak efficient solution.
		
		\item[(ii)] If condition (E1) does not hold, then the AKKT condition and the QNCQ do not guarantee the correctness of KKT optimality conditions even for smooth scalar optimization problems; see \cite[Example 4]{Andreani16} for more details.
		
		\item [(iii)]  Analysis similar to that in the proof of Theorem \ref{KKT-theorem} shows that  Theorems 4.1--4.5 in \cite{Giorgi16}   can be extended to multiobjective optimization problems with locally Lipschitz data. We leave the details to the reader.
	\end{enumerate}
\end{remark}

\section*{Acknowledgments}{The authors would like to thank the referees for their constructive comments which significantly improve the presentation of the paper. J.-C. Yao and C.-F. Wen are supported by the Taiwan MOST [grant number 106-2923-E-039-001-MY3], [grant number 106-2115-M-037-001],
	respectively, as well as the grant from Research Center for Nonlinear
	Analysis and Optimization, Kaohsiung Medical University, Taiwan.}


\end{document}